\documentclass[11pt]{amsart}

\usepackage[T1]{fontenc}
\usepackage[utf8]{inputenc}
\usepackage{amsmath,amssymb,amsthm,mathtools}
\usepackage[colorlinks=true,linkcolor=blue,citecolor=blue,urlcolor=blue]{hyperref}

\newtheorem{theorem}{Theorem}[section]
\newtheorem{proposition}[theorem]{Proposition}
\newtheorem{lemma}[theorem]{Lemma}
\newtheorem{corollary}[theorem]{Corollary}
\theoremstyle{definition}

\newcommand{\F}{\mathbf{F}}
\newcommand{\Int}{\operatorname{Int}}
\newcommand{\Frac}{\operatorname{Frac}}
\newcommand{\Spec}{\operatorname{Spec}}

\title[A Counterexample to Problem 19]{A Counterexample to Problem 19 on Integer-Valued Polynomial Rings}
\author{Haotian Ma}
\address{Zhejiang University}

\begin{document}

\begin{abstract}
We give a negative answer to Problem 19 of Cahen, Fontana, Frisch, and Glaz
concerning the flatness and freeness of rings of integer-valued polynomials.
We construct an explicit one-dimensional Noetherian local domain \(D\) over
the field with two elements and prove that the ring of integer-valued
polynomials on \(D\) is not flat as a \(D\)-module. The argument shows that a
certain polynomial is integer-valued on \(D\) with values in the integral
closure \(T\) of \(D\), but does not belong to the product of \(T\) with the
ring of integer-valued polynomials on \(D\). An application of Elliott's
flatness criterion then yields the counterexample. In particular, the ring of
integer-valued polynomials on an arbitrary integral domain need not be free.
The proof presented in this note was completed by Rethlas \cite{Rethlas2604},
a natural-language automated reasoning system; the author was
responsible for reviewing and checking the argument.
\end{abstract}

\maketitle

\section{Introduction}

Let \(D\) be an integral domain with quotient field \(K\). The ring of
integer-valued polynomials on \(D\) is
\[
  \Int(D)=\{f\in K[X] : f(D)\subseteq D\}.
\]
In the collection of open problems of Cahen, Fontana, Frisch, and Glaz
\cite{CFFG}, the following question is posed:

\begin{quote}
Is \(\Int(D)\) a flat \(D\)-module for every domain \(D\)? More generally, is
\(\Int(D)\) a free \(D\)-module for every domain \(D\)?
\end{quote}

The flatness question was motivated by several positive cases. If \(D\) is a
Dedekind domain, then \(\Int(D)\) is free as a \(D\)-module. More generally,
for Krull domains, and in fact for TV PvMD domains, \(\Int(D)\) is locally
free and hence flat; see \cite{Elliott,HoustonZafrullah,HwangChang}. The goal
of this note is to show that the general expectation is false.

\begin{theorem}\label{thm:main}
There exists a one-dimensional Noetherian local domain \(D\) such that
\(\Int(D)\) is not flat over \(D\). Consequently, \(\Int(D)\) need not be a
free \(D\)-module.
\end{theorem}

The domain is completely explicit:
\[
  D=\F_2+t(t+1)\,\F_2[t]_{\F_2[t]\setminus((t)\cup(t+1))}.
\]
The proof uses the following criterion of Elliott.

\section{Elliott's Flatness Criterion}

For an overring \(D'\) of \(D\), write
\[
  \Int(D,D')=\{f\in \Frac(D)[X] : f(D)\subseteq D'\}.
\]
If \(I\) is a nonzero fractional ideal of \(D\), we use the standard notation
\[
  I^{-1}=\{x\in \Frac(D) : xI\subseteq D\}.
\]

\begin{proposition}[Elliott {\cite[Proposition 2.13]{Elliott}}]
\label{prop:elliott}
Let \(D\) be an integral domain for which there exists a finitely generated
ideal \(I\) such that \(D'=I^{-1}\) is an overring of \(D\). If \(\Int(D)\) is
flat over \(D\), then
\[
  \Int(D,D')=D'\Int(D).
\]
\end{proposition}

We now construct a domain for which the equality in
Proposition~\ref{prop:elliott} fails.

\section{The Explicit Domain}

Let
\[
  k=\F_2,\qquad A=k[t],\qquad
  S=A\setminus\bigl((t)\cup(t+1)\bigr),
\]
and set
\[
  T=S^{-1}A,\qquad N_0=(t)T,\qquad N_1=(t+1)T,
\]
\[
  m=t(t+1),\qquad M=mT=N_0N_1,
\]
\[
  D=k+M=\{a+u : a\in k,\ u\in M\}\subseteq T.
\]

\begin{lemma}\label{lem:domain}
With the above notation, the following assertions hold.
\begin{enumerate}
\item \(T\) is a semilocal PID with maximal ideals \(N_0\) and \(N_1\), and
      \(M=N_0\cap N_1\).
\item \(D\) is a one-dimensional Noetherian local domain with maximal ideal
      \(M\) and residue field \(D/M\cong \F_2\).
\item One has \(T=D[t]\), the element \(t\) is integral over \(D\), and
      \(T\subseteq \Frac(D)\). In particular, \(\Frac(D)=\Frac(T)\).
\item \(T\) is the integral closure of \(D\) in \(\Frac(D)\), the conductor
      \((D:T)\) equals \(M\), and \(M^{-1}=T\).
\item The two maximal ideals \(N_0\) and \(N_1\) of \(T\) both contract to
      \(M\) in \(D\).
\end{enumerate}
\end{lemma}

\begin{proof}
Since \(T\) is obtained from the PID \(k[t]\) by inverting all elements outside
\((t)\cup(t+1)\), it is a semilocal PID whose maximal ideals are exactly
\[
  N_0=(t)T \qquad\text{and}\qquad N_1=(t+1)T.
\]
These ideals are comaximal, so
\[
  M=N_0N_1=N_0\cap N_1=mT.
\]
This proves (1).

We next show that \(D\) is local with maximal ideal \(M\). Since
\(D/M\cong k=\F_2\), the ideal \(M\) is maximal in \(D\). If \(u\in D\setminus
M\), then \(u=a+v\) with \(a\in \F_2^\times=\{1\}\) and \(v\in M\), so
\(u=1+v\). Because \(M\subseteq N_0\cap N_1\) lies in the Jacobson radical of
the semilocal ring \(T\), the element \(1+v\) is a unit of \(T\). Its inverse
lies in \(1+M\subseteq D\), so \(u\) is a unit of \(D\). Thus \(D\) is local
with maximal ideal \(M\).

To prove that \(T=D[t]\), let \(f(t)/g(t)\in T\) with \(g\in S\). Since
\(g\notin (t)\) and \(g\notin (t+1)\), we have \(g(0)=g(1)=1\) in \(\F_2\). It
follows that \(g-1\) vanishes at both \(0\) and \(1\), hence is divisible by
\(m=t(t+1)\) in \(k[t]\). Therefore
\[
  g=1+mh(t)
\]
for some \(h(t)\in k[t]\subseteq T\). Thus \(g\in 1+M\subseteq D\), and because
\(D\) is local, \(g\) is a unit of \(D\). Hence \(f(t)/g(t)\in D[t]\), proving
\(T=D[t]\).

The element \(t\) satisfies the monic equation
\[
  Y^2+Y+m=0,
\]
because \(t^2+t=m\in D\). Hence \(t\) is integral over \(D\), and since
\(T=D[t]\), the ring \(T\) is module-finite over \(D\). As \(T\) is Noetherian,
Eakin's theorem implies that \(D\) is Noetherian.

We now determine the dimension of \(D\). Let \(\mathfrak p\) be a nonzero prime
ideal of \(D\). Since \(T\) is integral over \(D\), there exists
\(\mathfrak q\in \Spec(T)\) with \(\mathfrak q\cap D=\mathfrak p\). Because
\(T\) is one-dimensional and \(\mathfrak q\neq 0\), the prime \(\mathfrak q\) is
maximal, so \(\mathfrak q\) is either \(N_0\) or \(N_1\). The contraction of a
maximal ideal under an integral extension is maximal, hence \(\mathfrak p\) is
maximal in \(D\). Since \(D\) is local, \(\mathfrak p=M\). Therefore \(M\) is
the only nonzero prime ideal of \(D\), so \(\dim D=1\). This completes (2).

For (3), note that \(m\in D\) and
\[
  mt=t^2(t+1)\in mT=M\subseteq D.
\]
Since \(m\neq 0\), it follows that
\[
  t=\frac{mt}{m}\in \Frac(D).
\]
Thus \(T=D[t]\subseteq \Frac(D)\), and therefore \(\Frac(D)=\Frac(T)\).

To prove the first statement in (4), let \(x\in \Frac(D)\) be integral over
\(D\). The same monic polynomial also shows that \(x\) is integral over \(T\),
and since \(T\) is integrally closed, we get \(x\in T\). Hence \(T\) is the
integral closure of \(D\) in \(\Frac(D)\).

For (5), take \(a+u\in D\) with \(a\in\F_2\) and \(u\in M\). Modulo \(N_0\) its
image is \(a\), because \(M\subseteq N_0\). Therefore \(a+u\in N_0\) if and
only if \(a=0\), namely if and only if \(a+u\in M\). Hence \(N_0\cap D=M\).
The same argument gives \(N_1\cap D=M\).

It remains to compute the conductor and \(M^{-1}\). Since \(MT\subseteq M\),
we have \(M\subseteq (D:T)\). Conversely, let \(y\in (D:T)\). Then
\(y=y\cdot 1\in D\), so write \(y=a+u\) with \(a\in\F_2\) and \(u\in M\). If
\(a=1\), then
\[
  yt=t+ut.
\]
Modulo \(N_0\) this element has residue \(0\), while modulo \(N_1\) it has
residue \(1\), because \(t\equiv 1 \pmod{N_1}\) and \(u\in M\subseteq N_1\).
Therefore \(yt\notin D\), contradicting \(y\in (D:T)\). Hence \(a=0\), so
\(y\in M\). Thus \((D:T)=M\).

Finally, if \(x\in T\), then \(xM\subseteq M\subseteq D\), so \(T\subseteq
M^{-1}\). Conversely, let \(x\in M^{-1}\). Then \(xm\in xM\subseteq D\), and
\[
  (xm)T=x(mT)=xM\subseteq D.
\]
Hence \(xm\in (D:T)=M=mT\). Write \(xm=ms\) with \(s\in T\). Since \(m\neq 0\)
in the domain \(T\), we obtain \(x=s\in T\). Thus \(M^{-1}=T\), completing
the proof.
\end{proof}

\section{An Explicit Obstruction Polynomial}

Define
\[
  f(X)=\frac{X^2+X}{m}\in \Frac(D)[X].
\]

\begin{proposition}\label{prop:obstruction}
With \(D\), \(T\), and \(m=t(t+1)\) as above, one has
\[
  f\in \Int(D,T)\setminus T\Int(D).
\]
Equivalently,
\[
  X^2+X\notin M\Int(D).
\]
\end{proposition}

\begin{proof}
Let \(x\in D\). Since \(D=\F_2+M\) and \(M=mT\), we may write
\[
  x=a+mv
\]
with \(a\in\F_2\) and \(v\in T\). Because \(a^2+a=0\) in \(\F_2\), we obtain
\[
  x^2+x=(a+mv)^2+(a+mv)=mv+m^2v^2=m(v+mv^2)\in mT.
\]
Therefore
\[
  f(x)=v+mv^2\in T,
\]
so \(f\in \Int(D,T)\).

Suppose now that \(f\in T\Int(D)\). Since \(mT=M\), this implies
\[
  X^2+X=mf\in m\,T\Int(D)=M\Int(D).
\]
Thus there exist \(a_1,\dots,a_r\in M\) and \(h_1,\dots,h_r\in \Int(D)\) such
that
\[
  X^2+X=\sum_{i=1}^r a_i h_i.
\]

Let \(v_1\) denote the discrete valuation on \(\Frac(T)\) corresponding to the
DVR \(T_{N_1}\), normalized by \(v_1(t+1)=1\). Choose an integer \(n\ge 1\)
strictly larger than the pole order at \(N_1\) of every coefficient of every
polynomial \(h_i\). Set
\[
  u=t(t+1)^{n+1}\in M.
\]
Consider a nonconstant term \(cX^r\) of some \(h_i\), with \(r\ge 1\). By the
choice of \(n\), we have \(v_1(c)\ge -(n-1)\), while
\[
  v_1(u)=n+1.
\]
Hence
\[
  v_1(cu^r)=v_1(c)+r\,v_1(u)\ge -(n-1)+r(n+1)\ge 2.
\]
Now write
\[
  h_i(X)-h_i(0)=\sum_{r\ge 1} c_r X^r,
\]
where each \(c_r\in \Frac(T)\). The preceding estimate shows that
\[
  v_1(c_r u^r)\ge 2
\]
for every \(r\ge 1\). Since \(h_i\in \Int(D)\) and \(u,0\in D\), we have
\[
  h_i(u),\,h_i(0)\in D\subseteq T.
\]
Therefore
\[
  h_i(u)-h_i(0)=\sum_{r\ge 1} c_r u^r
\]
is an element of \(T\). Applying the valuation \(v_1\) to this sum in
\(\Frac(T)\), we obtain
\[
  v_1\bigl(h_i(u)-h_i(0)\bigr)
  =v_1\!\left(\sum_{r\ge 1} c_r u^r\right)
  \ge \min_{r\ge 1} v_1(c_r u^r)\ge 2.
\]
Since \(h_i(u)-h_i(0)\in T\) and its \(v_1\)-value is strictly positive, it
follows that
\[
  h_i(u)-h_i(0)\in N_1.
\]
Now \(u,0\in D\), so \(h_i(u),h_i(0)\in D\). Since \(N_1\cap D=M\) by
Lemma~\ref{lem:domain}, it follows that
\[
  h_i(u)-h_i(0)\in M.
\]
Because each \(a_i\in M\), we conclude that
\[
  a_i h_i(u)-a_i h_i(0)\in M^2
\]
for every \(i\). Summing over \(i\), we obtain
\[
  (u^2+u)-0=(X^2+X)(u)-(X^2+X)(0)\in M^2.
\]
Since \(u\in M\), we also have \(u^2\in M^2\), whence
\[
  u\equiv u^2+u \pmod{M^2}.
\]
Thus \(u\in M^2\).

This is impossible. Indeed, let \(v_0\) be the discrete valuation corresponding
to \(T_{N_0}\), normalized by \(v_0(t)=1\). Then
\[
  v_0(u)=v_0\bigl(t(t+1)^{n+1}\bigr)=1,
\]
whereas every element of \(M^2=m^2T\) has \(v_0\)-value at least \(2\). This
contradiction shows that
\[
  X^2+X\notin M\Int(D),
\]
and hence \(f\notin T\Int(D)\).
\end{proof}

\section{The Counterexample}

\begin{proof}[Proof of Theorem~\ref{thm:main}]
Let \(D\) be the domain constructed in Section~3. By
Lemma~\ref{lem:domain}, \(D\) is a one-dimensional Noetherian local domain
with maximal ideal \(M\), and \(M^{-1}=T\) is an overring of \(D\). Since
\(D\) is Noetherian, the ideal \(M\) is finitely generated.

By Proposition~\ref{prop:obstruction}, the polynomial
\[
  f(X)=\frac{X^2+X}{t(t+1)}
\]
lies in \(\Int(D,T)\) but not in \(T\Int(D)\). Therefore
\[
  \Int(D,T)\neq T\Int(D).
\]
If \(\Int(D)\) were flat over \(D\), then
Proposition~\ref{prop:elliott} applied to \(I=M\) and \(D'=M^{-1}=T\) would
imply the opposite equality
\[
  \Int(D,T)=T\Int(D),
\]
a contradiction. Hence \(\Int(D)\) is not flat over \(D\).

Finally, every free module is flat, so \(\Int(D)\) cannot be free over \(D\).
\end{proof}

\begin{corollary}
Problem \(19\mathrm{a}\) and Problem \(19\mathrm{b}\) of \cite{CFFG} both have
negative answers.
\end{corollary}

\section{Use of AI}

The proof in this note was obtained with the assistance of Rethlas
\cite{Rethlas2604}, a natural-language automated reasoning system for
automated conjecture resolution. In the present work, Rethlas was used as a
natural-language mathematical reasoning tool for searching for a
counterexample, organizing the construction, and producing a candidate proof.

For this problem, Rethlas supplied the explicit domain \(D=k+M\), the
connection with Elliott's flatness criterion, and the obstruction showing that
the equality \(\Int(D,T)=T\Int(D)\) fails. The proof in its entirety was then
verified by the human author.

\end{document}